\documentclass[12pt]{amsart}

\usepackage{xcolor}

\definecolor{norange}{rgb}{0.898, 0.621, 0.0}
\definecolor{skyblue}{rgb}{0.336, 0.703, 0.910}
\definecolor{bluishgreen}{rgb}{0, 0.617, 0.449}
\definecolor{nyellow}{rgb}{0.937, 0.890, 0.258}
\definecolor{nblue}{rgb}{0, 0.445, 0.695}
\definecolor{nred}{rgb}{0.832, 0.367, 0}
\definecolor{npurple}{rgb}{0.797, 0.473, 0.652}

\usepackage{amscd,amsmath,amssymb,amsthm,amsfonts,comment}

\usepackage[hidelinks]{hyperref}
\usepackage{tikz}
\tikzstyle{vertex}=[circle, draw, fill=black, inner sep=0pt, minimum size=3pt]
\usetikzlibrary{decorations.text}
\usetikzlibrary{patterns}
\usetikzlibrary{positioning}

\newcommand{\Rr}{\mathbb{R}}

\newcommand\ideal[1]{\left< #1 \right>}
\newcommand\set[1]{\{ #1 \}}
\newcommand\abs[1]{\left| #1 \right|}
\newcommand\parens[1]{\left( #1 \right)}

\DeclareMathOperator{\link}{link}
\DeclareMathOperator{\lk}{link}

\DeclareMathOperator{\init}{init}

\newtheorem{theorem}{Theorem}[section]

\newtheorem{example}[theorem]{Example}

\newtheorem{proposition}[theorem]{Proposition}

\newtheorem{definition}[theorem]{Definition}

\newtheorem{lemma}[theorem]{Lemma}

\newtheorem{corollary}[theorem]{Corollary}

\newtheorem{conjecture}[theorem]{Conjecture}

\newtheorem{question}{Question}

\newtheorem{remark}[theorem]{Remark}

\makeatletter
\renewcommand*\env@matrix[1][*\c@MaxMatrixCols c]{%
  \hskip -\arraycolsep
  \let\@ifnextchar\new@ifnextchar
  \array{#1}}
\makeatother

\title{Resolving Stanley's conjecture on $k$-fold acyclic complexes}

\author{Joseph Doolittle}
\address{Institut f\"{u}r Geometrie, Technische Universit\"{a}t Graz}
\email{jdoolittle@tugraz.at}

\author{Bennet Goeckner}
\address{Department of Mathematics, University of Washington}
\email{goeckner@uw.edu}

\begin{document}

\begin{abstract}
In 1993 Stanley showed that if a simplicial complex is acyclic over some field, then its face poset can be decomposed into disjoint rank $1$ boolean intervals whose minimal faces together form a subcomplex. Stanley further conjectured that complexes with a higher notion of acyclicity could be decomposed in a similar way using boolean intervals of higher rank. We provide an explicit counterexample to this conjecture. We also prove a version of the conjecture for boolean trees and show that the original conjecture holds when this notion of acyclicity is as high as possible.
\end{abstract}

\maketitle

\section{Introduction}

The interplay between combinatorial and topological properties of simplicial complexes has been a subject of great interest for researchers for many decades (see, e.g.,  \cite{Bj95, KN16, St96}). One particularly beautiful result due to Stanley  connects the homology of the geometric realization of a complex to a well-behaved decomposition of its face poset. 

\begin{theorem}\label{StanleyTheorem}\emph{\cite[Theorem 1.2]{St93}}
Let \(\Delta\) be a simplicial complex that is acyclic over some field $\Bbbk$. The face poset of \(\Delta\) can be written as the disjoint union of rank $1$  boolean intervals such that the minimal faces of these intervals together form a subcomplex of $\Delta$.
\end{theorem}

This theorem was generalized by Stanley \cite[Proposition 2.1]{St93} and Duval \cite[Theorem 1.1]{Du94}. Stanley further conjectured \cite[Conjecture 2.4]{St93} that complexes with a higher notion of acyclicity possess similar decompositions into boolean intervals of higher rank.

\begin{definition}\label{kFoldDef}
A simplicial complex is \emph{\(k\)-fold acyclic} if $\link_\Delta \sigma$ is acyclic (over a field $\Bbbk$) for all $\sigma \in \Delta$ such that $\abs{\sigma} < k$.
\end{definition}

Observe that $1$-fold acyclicity is equivalent to acyclicity, so the following conjecture seeks to extend Theorem~\ref{StanleyTheorem}.

\begin{conjecture}\label{DiamondConjecture}\cite[Conjecture 2.4]{St93}
Let $\Delta$ be a \(k\)-fold acyclic simplicial complex. Then $\Delta$ can be decomposed into disjoint rank $k$ boolean intervals, the minimal faces of which together form a subcomplex.
\end{conjecture}

Conjecture~\ref{DiamondConjecture} also appears as \cite[Problem 27]{St96} and \cite[Problem 5]{St99}.

Central to this problem is the $f$-polynomial $f(\Delta,t)$ of a $d$-dimensional simplicial complex
$$
f(\Delta,t) = \sum_{\sigma \in \Delta} t^{\abs{\sigma}} = f_{-1} + f_0 t + f_1 t^2 + \dots + f_d t^{d+1}
$$
where $f_i = f_i(\Delta)$ is the number of $i$-dimensional faces of $\Delta$. 
Theorem~\ref{StanleyTheorem} shows that if $\Delta$ is acyclic, then $f(\Delta,t) = (1+t)f(\Gamma,t)$ where $\Gamma$ is a subcomplex. Earlier, Kalai \cite{Ka85} showed this equality holds for some complex \(\Gamma\), which is not necessarily a subcomplex of \(\Delta\). Using results from Kalai's algebraic shifting \cite{Ka00}, Stanley \cite[Proposition 2.3]{St93} further showed that the $f$-polynomial of a \(k\)-fold acyclic complex can be written as
\begin{align}\label{kFoldfPoly}
f(\Delta,t) = (1+t)^kf(\Gamma,t)
\end{align}
for some complex $\Gamma$ (which is not necessarily a subcomplex of $\Delta$). If Conjecture~\ref{DiamondConjecture} were true, it would provide a combinatorial witness for this $\Gamma$. We provide an explicit counterexample to Conjecture~\ref{DiamondConjecture}. We also prove a weaker version of the original conjecture in Theorem~\ref{AcycBoolTreeThm}.

In Section~\ref{background}, we review definitions and relevant background material. In Section~\ref{counterexample}, we describe our construction of the counterexample to Conjecture~\ref{DiamondConjecture}, which relies on reducing the problem to relative complexes and follows ideas similar to those recently developed by Duval et al.~\cite{DG16} and Juhnke-Kubitzke and Venturello \cite{JV18}. In Section~\ref{bennet}, we prove a weaker version of Conjecture~\ref{DiamondConjecture}, replacing boolean intervals with boolean trees. This result gives a witness for the behavior observed in \eqref{kFoldfPoly}. In Section~\ref{dimfold}, we show that Conjecture~\ref{DiamondConjecture} holds when $k$ is equal to the dimension of the complex. We end with a section of open questions.

\section{Preliminaries}\label{background}

We let $[n]$ denote the set $\{1,\dots,n\}$. A \emph{simplicial complex} $\Delta$ on $[n]$ is a subset of $2^{[n]}$ such that if $\sigma \in \Delta$ and $\tau \subseteq \sigma$, then $\tau \in \Delta$. The elements of $\Delta$ are \emph{faces}, and maximal faces are \emph{facets}. If $F_1,\dots,F_j$ are the facets of $\Delta$, we will often write $\Delta = \ideal{F_1,\dots,F_j}$, since the facets uniquely determine \(\Delta\). The \emph{dimension} of a face $\sigma$ is $\dim \sigma = \abs{\sigma}-1$ and the dimension of $\Delta$ is $\dim \Delta = \max\{ \dim \sigma \mid \sigma \in \Delta \}.$ A complex is \emph{pure} if all facets have the same dimension. For a pure complex, a \emph{ridge} is a face of one dimension lower than the facets. Unless otherwise specified, we assume throughout that $\dim \Delta = d.$

A \emph{subcomplex} of $\Delta$ is a simplicial complex $\Gamma$ such that $\Gamma \subseteq \Delta$. If $W \subseteq [n]$, then the \emph{induced subcomplex} on $W$ is $\Delta|_W := \{ \sigma \in \Delta ~\mid~ \sigma \subseteq W \}.$  Given a face $\sigma \in \Delta,$ the \emph{link} of $\sigma$ in $\Delta$ is
\begin{align*}
\link_\Delta \sigma &= \{ \tau \in \Delta \mid \tau \cup \sigma \in \Delta  , \tau \cap \sigma = \varnothing\},
\end{align*}
which we will often denote as $\link \sigma$ if there is no possibility of confusion. Given two complexes $\Delta_1$ and $\Delta_2$ on disjoint vertex sets, their \emph{join} is $\Delta_1 \star \Delta_2 = \{ \sigma_1 \cup \sigma_2 \mid \sigma_1 \in \Delta_1, \sigma_2 \in \Delta_2 \}$. If $\Delta_1$ is a $(k-1)$-simplex, then this join is the \emph{$k$-fold cone} of $\Delta_2$.

A complex is \emph{acyclic} (over a field $\Bbbk$) if all of its reduced homology groups are trivial. Recalling Definition~\ref{kFoldDef}, we note that $1$-fold acyclicity is equivalent to acyclicity, so Theorem~\ref{StanleyTheorem} is the $k=1$ case of Conjecture~\ref{DiamondConjecture}. When $k>1$, $k$-fold acyclicity is not a topological property. For example, the $d$-simplex is $(d+1)$-fold acyclic, but its barycentric subdivision is not $k$-fold acyclic for $k>1$.

The construction of our counterexample relies on \emph{relative simplicial complexes}; given a simplicial complex $\Delta$ and a subcomplex $\Gamma$, the relative complex $\Phi = (\Delta, \Gamma)$ is the set of all of the faces of $\Delta$ that are not faces of $\Gamma$.

Given a poset $P$ and two elements $x,y \in P$, the \emph{interval} from $x$ to $y$ is $[x,y] = \{ z \in P \mid x \le z \le y \}.$ If $[x,y] = \set{x,y}$, then we say that $y$ \emph{covers} $x$. An interval $I$ is a \emph{rank $k$ boolean interval} if $I \cong 2^{[k]}$. A \emph{boolean interval decomposition} of $P$ is a collection $\mathcal{B}$ of disjoint boolean intervals in $P$ such that
$$
P = \bigsqcup_{I \in \mathcal{B}} I.
$$
Such a decomposition is a \emph{rank $k$ boolean interval decomposition} if all intervals in the decomposition are of rank $k$. We also refer to this as a rank \(k\) boolean decomposition.

\begin{definition}
A \emph{boolean tree} of rank $i$ is a subposet $T_i$ of a poset $P$, that has a unique minimal element $r$, and is defined recursively as follows. Any subposet with exactly one element is a boolean tree of rank 0. Now assume $T_1$ and $T_2$ are two disjoint boolean trees of rank $(i - 1)$, each with minimal elements $r_1$ and $r_2$ respectively, such that $r_2$ covers $r_1$ in $P$. Then $T_1 \cup T_2$ is a boolean tree of rank $i$, with  $r_1$ as its unique minimal element.
\end{definition}

A \emph{(rank $k$) boolean tree decomposition} of a poset is defined the same as a (rank $k$) boolean interval decomposition, except that boolean intervals are replaced with boolean trees. When the rank is \(1\), both boolean trees and boolean intervals are simply described by a single covering relation, matching Stanley's original theorem statement \cite[Theorem 1.2]{St93}.

\begin{definition}\label{StackedDef}
A simplicial complex $\Delta$ is a \emph{stacked simplicial complex} if $\Delta$ is pure of dimension \(d\) with a facet order $F_1,\dots,F_j$ such that for each \(i \in [j-1]\), $\ideal{F_1,\dots,F_i}\cap \ideal{F_{i+1}}$ is a \((d-1)\)-simplex. Such an order is known as a \emph{stacked shelling}.
\end{definition}

Another characterization of stacked complexes, equivalent to Definition~\ref{StackedDef}, is that $\Delta$ is stacked if and only if the minimal new face added by \(F_i\) is a vertex, for \(i \neq 1\). That is, a pure complex $\Delta$ is stacked if and only if there exists an order on its facets $F_1,\dots,F_j$ such that for each \(i \in [j-1]\), $\ideal{F_{i+1}} \setminus \ideal{F_1,\dots,F_i} = [v_{i+1},F_{i+1}]$ for some vertex $v_{i+1}$. We note that stacked complexes are Cohen--Macaulay quasi-forests (see, e.g., \cite[Section~9.2.3]{HH}).

\section{Construction}\label{counterexample}

To construct our counterexample, we first show that we can reduce the problem to finding a relative complex $(\Delta,\Gamma)$ with appropriate properties. Then we construct a relative pair $(\Delta,\Gamma)$ that meets the necessary requirements and use it to construct the counterexample to Conjecture~\ref{DiamondConjecture}. This technique of reducing to a relative complex has been used successfully by Duval et al. \cite{DG16} and Juhnke-Kubitzke and Venturello \cite{JV18}.

\begin{lemma}\label{AcycGlueProp}
Let $\Delta_1$ and $\Delta_2$ be simplicial complexes such that $\Delta_1$ is $j$-fold acyclic, $\Delta_2$ is $k$-fold acyclic, and $\Delta_1 \cap \Delta_2$ is $\ell$-fold acyclic. Then $\Delta_1 \cup \Delta_2$ is $m$-fold acyclic, where $m = \min\set{j,k,\ell}$.
\end{lemma}

\begin{proof} 
Let $\abs{\sigma} < m$ and assume $\sigma \in \Delta_1 \cup \Delta_2$. If $\sigma \in \Delta_1 \setminus \Delta_2,$ then $\link_{\parens{\Delta_1 \cup \Delta_2}} \sigma = \link_{\Delta_1} \sigma$ and thus $\link_{\parens{\Delta_1 \cup \Delta_2}} \sigma$ is acyclic. The same holds if $\sigma \in \Delta_2 \setminus \Delta_1.$

If instead $\sigma \in \Delta_1 \cap \Delta_2$, then we note that $\link_{\Delta_1 \cup \Delta_2} \sigma = \link_{\Delta_1} \sigma \cup \link_{\Delta_2} \sigma$ and similarly $\link_{\Delta_1 \cap \Delta_2} \sigma = \link_{\Delta_1} \sigma \cap \link_{\Delta_2} \sigma$. We then have the Mayer-Vietoris sequence
$$
\dots \to \tilde{H}_i(\link_{\Delta_1} \sigma ) \oplus \tilde{H}_i(\link_{\Delta_2} \sigma ) \to \tilde{H}_i(\link_{\Delta_1 \cup \Delta_2} \sigma) \to \tilde{H}_{i-1}(\link_{\Delta_1 \cap \Delta_2} \sigma ) \to \dots .
$$
Since $\Delta_1$, $\Delta_2$, and $\Delta_1 \cap \Delta_2$ are $m$-fold acyclic, the homology groups of the links of $\sigma$ in each of these complexes vanish since  $\abs{\sigma} < m$. This implies that $\tilde{H}_i(\link_{\Delta_1 \cup \Delta_2} \sigma) = 0$ for all $i$.  Therefore $\Delta_1 \cup \Delta_2$ is $m$-fold acyclic.
\end{proof}

Lemma~\ref{AcycGlueProp} is used to preserve $k$-fold acyclicity in the following theorem.

\begin{theorem}\label{GoodAcycUnicornSoup}
Let $\Phi = (\Delta,\Gamma)$ be a relative complex such that
\begin{enumerate}
\item $\Delta$ and $\Gamma$ are $k$-fold acyclic;
\item $\Gamma$ is an induced subcomplex of $\Delta$; and
\item $\Phi$ cannot be written as a disjoint union of rank $k$ boolean intervals.
\end{enumerate}
Let $\ell$ be the total number of faces of $\Gamma$ and let $N>\ell/2^k$. If $\Omega = \Omega_N$ is the complex formed by gluing $N$ copies of $\Delta$ together along $\Gamma$, then $\Omega$ is a $k$-fold acyclic complex that cannot be written as a disjoint union of rank $k$ boolean intervals.
\end{theorem}
\begin{proof}

Since $\Gamma$ is an induced subcomplex of $\Delta$, gluing copies of $\Delta$ together along $\Gamma$ will result in a simplicial complex. By Lemma~\ref{AcycGlueProp}, this resulting complex $\Omega$ is $k$-fold acyclic. 
The face poset of $\Omega$ is precisely $N$ disjoint copies of $\Phi$ and one copy of $\Gamma$. We note that there are at most $ \ell/2^k$ disjoint rank \(k\) boolean intervals in $\Gamma$. 

Since each of \(\Delta\) and \(\Gamma\) is \(k\)-fold acyclic, their $f$-polynomials are each divisible by $(1+t)^k$. This implies that the $f$-polynomial of \(\Phi\) is \(p(t)(1+t)^k\), for some polynomial \(p(t)\). Therefore \(\Phi\) has \(p(1)2^k\) many faces. Since $\Phi$ cannot be decomposed using only $p(1)$ boolean intervals, a collection of disjoint rank \(k\) boolean intervals in the face poset of \(\Omega\) which covers \(\Phi\) must consist of \(b>p(1)\) intervals. Such a collection must also contain at least \(2^k\) faces of \(\Omega\) which are not in \(\Phi\). Since the copies of \(\Phi\) in \(\Omega\) are incomparable, these faces must be in \(\Gamma\).

There are at most $\ell/2^k$ mutually disjoint collections of \(2^k\) faces in $\Gamma$, and there are $N>\ell/2^k$ copies of $\Phi$ in $\Omega.$ By the pigeonhole principle, any rank $k$ boolean decomposition of $\Omega$ must contain some \(\Phi\) which is decomposed into disjoint rank \(k\) boolean intervals. This is a contradiction, so $\Omega$ is not rank $k$ boolean decomposable.
\end{proof}

We now start the construction of our counterexample, beginning with the following relative complex $\Psi$, which is isomorphic to the complex in \cite[Remark 3.6]{DG16}. We have shortened the notation for faces so instead of writing $\set{1,2,3,4}$ we write $1234$, for example.

\begin{example}\label{RelativeExample}
Let $\Sigma$ and $\Upsilon$ be the following simplicial complexes and let \(\Psi\) be the relative complex between them.
\begin{align*}
\Sigma &= \ideal{1234,1235,2345,2456,3456}, \\
\Upsilon &= \ideal{125,124,246,346}, \\
\Psi &= (\Sigma,\Upsilon).
\end{align*}
Both $\Sigma$ and $\Upsilon$ are $2$-fold acyclic
and the face poset of $\Psi$ cannot be decomposed into disjoint rank 2 boolean intervals. The face poset of \(\Psi\) appears in Figure~\ref{psiposet} for the reader to verify this.

\begin{figure}[t]
\begin{center}
\begin{tikzpicture}
\node[circle,draw=black, fill=white, inner sep=1pt,minimum size=5pt] at (1,4) (1234){$1234$};
\node[circle,draw=black, fill=white, inner sep=1pt,minimum size=5pt] at (2.75,4) (1235){$1235$};
\node[circle,draw=black, fill=white, inner sep=1pt,minimum size=5pt] at (4.5,4) (2345){$2345$};
\node[circle,draw=black, fill=white, inner sep=1pt,minimum size=5pt] at (8,4) (2456){$2456$};
\node[circle,draw=black, fill=white, inner sep=1pt,minimum size=5pt] at (6.25,4) (3456){$3456$};
\node[circle,draw=black, fill=white, inner sep=1pt,minimum size=5pt] at (1,2) (123){$123$};
\node[circle,draw=black, fill=white, inner sep=1pt,minimum size=5pt] at (0,2) (134){$134$};
\node[circle,draw=black, fill=white, inner sep=1pt,minimum size=5pt] at (3,2) (234){$234$};
\node[circle,draw=black, fill=white, inner sep=1pt,minimum size=5pt] at (2,2) (135){$135$};
\node[circle,draw=black, fill=white, inner sep=1pt,minimum size=5pt] at (4,2) (235){$235$};
\node[circle,draw=black, fill=white, inner sep=1pt,minimum size=5pt] at (6,2) (245){$245$};
\node[circle,draw=black, fill=white, inner sep=1pt,minimum size=5pt] at (5,2) (345){$345$};
\node[circle,draw=black, fill=white, inner sep=1pt,minimum size=5pt] at (9,2) (256){$256$};
\node[circle,draw=black, fill=white, inner sep=1pt,minimum size=5pt] at (8,2) (456){$456$};
\node[circle,draw=black, fill=white, inner sep=1pt,minimum size=5pt] at (7,2) (356){$356$};
\node[circle,draw=black, fill=white, inner sep=1pt,minimum size=5pt] at (1,0) (13){$13$};
\node[circle,draw=black, fill=white, inner sep=1pt,minimum size=5pt] at (8,0) (56){$56$};
\node[circle,draw=black, fill=white, inner sep=1pt,minimum size=5pt] at (2.75,0) (23){$23$};
\node[circle,draw=black, fill=white, inner sep=1pt,minimum size=5pt] at (6.25,0) (45){$45$};
\node[circle,draw=black, fill=white, inner sep=1pt,minimum size=5pt] at (4.5,0) (35){$35$};
\draw (1234) -- (123) -- (13) -- (134) -- (1234) -- (234) -- (23) -- (235) -- (2345) -- (245) -- (45) -- (345) -- (3456) -- (456) -- (45);
\draw (1235) -- (135) -- (35) -- (345) -- (2345) -- (234);
\draw (23) -- (123) -- (1235) -- (235) -- (35) -- (356);
\draw (456) -- (56) -- (356) -- (3456);
\draw (56) -- (256) -- (2456) -- (245);
\draw (13) -- (135);
\draw (2456) --(456);
\end{tikzpicture}
\end{center}
\caption{The face poset of \(\Psi = (\Sigma, \Upsilon)\) in Example~\ref{RelativeExample}.}\label{psiposet}
\end{figure}

\end{example}

Since $\Upsilon$ is not an induced subcomplex of $\Sigma$, we cannot immediately apply Theorem~\ref{GoodAcycUnicornSoup} to produce a counterexample to Conjecture~\ref{DiamondConjecture}. However, this complex is the foundation of our counterexample and will be referred to repeatedly in our construction.

\begin{figure}
\begin{center}
\begin{tikzpicture}

\definecolor{norange}{rgb}{0.898, 0.621, 0.0}
\definecolor{skyblue}{rgb}{0.336, 0.703, 0.910}
\definecolor{bluishgreen}{rgb}{0, 0.617, 0.449}
\definecolor{nyellow}{rgb}{0.937, 0.890, 0.258}
\definecolor{nblue}{rgb}{0, 0.445, 0.695}
\definecolor{nred}{rgb}{0.832, 0.367, 0}
\definecolor{npurple}{rgb}{0.797, 0.473, 0.652}

\node[coordinate] at (0,0) (A){};
\node[coordinate] at (0,1) (B){};
\node[coordinate] at (4,-2) (C){};
\node[coordinate] at (4,-1) (D){};
\node[coordinate] at (-3,-2) (E){};
\node[coordinate] at (-3,-1) (F){};
\node[coordinate] at (1,3) (G){};
\node[coordinate] at (1,4) (H){};
\draw[fill=nred] (E) -- (F) -- (G) -- cycle;
\draw[fill=nred] (F) -- (G) -- (H) -- cycle;
\draw[fill=nblue] (C) -- (D) -- (H) -- cycle;
\draw[fill=nblue] (C) -- (G) -- (H) -- cycle;
\draw[fill=npurple] (A) -- (B) -- (G) -- cycle;
\draw[fill=npurple] (H) -- (B) -- (G) -- cycle;
\draw[fill=bluishgreen] (A) -- (B) -- (C) -- cycle;
\draw[fill=bluishgreen] (D) -- (B) -- (C) -- cycle;
\draw[fill=skyblue] (A) -- (B) -- (E) -- cycle;
\draw[fill=skyblue] (F) -- (B) -- (E) -- cycle;
\draw[fill=nyellow] (C) -- (D) -- (F) -- cycle;
\draw[fill=nyellow] (C) -- (E) -- (F) -- cycle;
\node[vertex] at (A) {};
\node[vertex] at (B) {};
\node[vertex] at (C) {};
\node[vertex] at (D) {};
\node[vertex] at (E) {};
\node[vertex] at (F) {};
\node[vertex] at (G) {};
\node[vertex] at (H) {};
\node[below] at (A) {$A$};
\node[above left=-1.5pt and -1.5pt] at (B) {$B$};
\node[right] at (C) {$C$};
\node[right] at (D) {$D$};
\node[left] at (E) {$E$};
\node[left] at (F) {$F$};
\node[below right=5pt and -6pt] at (G) {$G$};
\node[above] at (H) {$H$};
\end{tikzpicture}
\end{center}
\caption{The triangles listed in \eqref{TrianglePairs}. Each hollow triangular prism that appears is filled in by three facets of \(\Gamma\).}\label{fig:Triangles}
\end{figure}

Our goal is to create a new pair $(\Delta,\Gamma)$  that meets 
the conditions of Theorem~\ref{GoodAcycUnicornSoup}. 
 We now consider the following complex, $\Gamma$.
\[
\Gamma = \left< ABCE, BCEF, BCDF, ABCG, BCGH, BCDH, ABEG, BEFG, BFGH \right>
\]
Since $\Gamma$ is a simplicial $3$-ball with no interior vertices, $\Gamma$ is $2$-fold acyclic. Within \(\Gamma\) there are six pairs of triangles which are listed below and depicted in Figure~\ref{fig:Triangles}.
\begin{equation}\label{TrianglePairs}
\begin{matrix}[lll]
\{ABC,BCD\}, & \{ABE,BEF\}, & \{ABG,BGH\}, \\
\{CDF,CEF\}, & \{CDH,CGH\}, & \{EFG,FGH\}.
\end{matrix}
\end{equation}
We now begin the construction of $\Delta$. To each of the edges $AB,$ $CD,$ $EF,$ and $GH$, we add cone points $I,$ $J$, $K$, and $L$ respectively, forming four triangles that are not in \(\Gamma\):
\begin{equation}\label{FourTriangles}
ABI,CDJ,EFK,GHL.
\end{equation}

For any pair of triangles from \eqref{FourTriangles} there is a unique pair of triangles in \eqref{TrianglePairs} so that the four triangles together form a complex isomorphic to \(\Upsilon\) from Example~\ref{RelativeExample}. For example, the two triangles \(\{ABI, CDJ\}\) from \eqref{FourTriangles} together with \(\{ABC,BCD\}\) from \eqref{TrianglePairs} form a complex isomorphic to \(\Upsilon\). For these four triangles, we will glue a copy of \(\Sigma\) to \(\Gamma\) along this \(\Upsilon\) in the natural way. 

We obtain \(\Delta\) as the result of gluing six copies of \(\Sigma\) to \(\Gamma\) in this way, one for each choice of two triangles from \eqref{FourTriangles}. For clarity, we list all of the facets of \(\Delta\) that are not in \(\Gamma\). 
\begin{equation}\label{OtherFacetsOfDelta}
\begin{matrix}[lllll]
            ABCJ,& ABIJ,& BCIJ,& BCDI,& CDIJ,\\ 
            ABEK,& ABIK,& BEIK,& BEFI,& EFIK,\\ 
            ABGL,& ABIL,& BGIL,& BGHI,& GHIL,\\ 
            CDFK,& CDJK,& CFJK,& CEFJ,& EFJK,\\ 
            CDHL,& CDJL,& CHJL,& CGHJ,& GHJL,\\ 
            EFGL,& EFKL,& FGKL,& FGHK,& GHKL.\\ 
\end{matrix}
\end{equation}
It is straightforward to verify that \(\Delta\) is 2-fold acyclic (see our Sage code \cite{linprog}) and that \(\Gamma\) is an induced subcomplex of \(\Delta\). It only remains to be shown that $(\Delta,\Gamma)$ is not decomposable into rank $2$ boolean intervals; then we can apply Theorem~\ref{GoodAcycUnicornSoup} to construct our counterexample.

\begin{theorem}\label{Not2Decomposable}
 $\Phi = (\Delta,\Gamma)$ is not rank $2$ boolean decomposable.
\end{theorem}
The following proof follows a similar structure to the proof of Theorem~\ref{GoodAcycUnicornSoup}.
\begin{proof} 
The face poset of $\Phi$ contains six copies of the face poset of $\Psi$, one for each of the copies of $\Sigma$ that was glued in above. Each of the six copies of $\Psi$ is not rank $2$ boolean decomposable, so each requires at least $4$ additional faces to possibly be rank $2$ boolean decomposable, a total of \(24\) faces. These six copies of $\Psi$ are pairwise disjoint and pairwise incomparable. The only faces of $\Phi$ not contained in a copy of $\Psi$ are the intervals $[I,ABI]$, $[J,CDJ]$, $[K,EFK]$, and $[L,GHL]$.  These \(16\) additional faces are not sufficient for every copy of \(\Psi\) to become rank \(2\) boolean decomposable, so \(\Phi\) is not decomposable into rank $2$ boolean intervals.
\end{proof}

Since $f(\Gamma) = (1, 8, 22, 24, 9)$, Theorem~\ref{GoodAcycUnicornSoup} immediately implies that $\Omega_{17}$ is a counterexample to Conjecture~\ref{DiamondConjecture}, since \(17>64/4\). 

In the proof of Theorem~\ref{Not2Decomposable}, $\Phi$ had only 16 faces that were not included in a copy of $\Psi$ instead of the 24 needed to possibly become rank 2 boolean decomposable. This difference of 8 is greater than the 4 assumed in Theorem~\ref{GoodAcycUnicornSoup}. Furthermore, only the 43 faces of \(\Gamma\) that are below \(\Phi\) can be used to make \(\Phi\) rank 2 boolean decomposable. With these improvements, we know that \(\Omega_6\) is a smaller counterexample, since \(6>43/8\). In fact, we can find an even smaller counterexample.

\begin{remark}\label{SmallCounterexample}
A linear program \cite{linprog} verifies that 
$\Omega = \Omega_3$ is a counterexample to Conjecture~\ref{DiamondConjecture}. The f-polynomial of this counterexample is \( f(\Omega_3,t) = 1+20t+136t^2+216t^3+99t^4 = (1+t)^2(1+18t+99t^2)\). This is the smallest known counterexample to Conjecture~\ref{DiamondConjecture}.
\end{remark}

Considering the above discussion showing that $\Omega_6$ is a counterexample to Conjecture~\ref{DiamondConjecture}, one might look to find a straightforward proof of Remark~\ref{SmallCounterexample} that does not rely on a computer calculation.

\begin{proposition}\label{Prop:Cone}
Let $\Delta$ be a simplicial complex. Then $\Delta$ has a rank $k$ boolean decomposition if and only if the cone $\ideal{v} \star \Delta$ has a rank $(k+1)$ boolean decomposition.
\end{proposition}

\begin{proof}
If $\Delta$ has a rank $k$ boolean decomposition, then each interval \([\sigma,\tau]\) in this decomposition can be modified into the interval \([\sigma, \tau \cup \{v\}]\) to form a rank $(k+1)$ boolean decomposition of $\ideal{v} \star \Delta$.

Assume instead that $\Delta$ does not have a rank $k$ boolean decomposition but that $\ideal{v} \star \Delta$ does have a rank $(k+1)$ boolean decomposition. Consider an interval $[\sigma,\tau]$ in this decomposition. If $v \in \sigma,\tau$, none of these faces appear in $\Delta$. If $v \in \tau$ but $v \notin \sigma$, $[\sigma,\tau]$ can be decomposed into $[\sigma, \tau \setminus \{v\}]$ and $[\sigma \cup \{v\}, \tau]$, where the first is a rank $k$ interval of $\Delta$ and the second contains no faces of $\Delta$. If instead $v \notin \sigma,\tau$, then $[\sigma,\tau]$ can be decomposed into $[\sigma, \tau \setminus \{x\}]$ and $[\sigma \cup \{x\}, \tau]$ for any $x \in \tau$. These intervals together give a decomposition of $\Delta$ into rank $k$ boolean intervals, which is a contradiction.
\end{proof}

Since the cone of a $k$-fold acyclic complex is $(k+1)$-fold acyclic, Proposition~\ref{Prop:Cone} allows us to produce higher-dimensional counterxamples to Conjecture~\ref{DiamondConjecture}.

\section{Boolean Trees}\label{bennet}

While Conjecture~\ref{DiamondConjecture} is false, we will use this section to prove a weakened version of it by replacing boolean intervals with boolean trees. We will rely on algebraic shifting, developed by Kalai in \cite{Ka00} and iterated homology, developed by Duval and Rose in \cite{DR00} and Duval and Zhang in \cite{DZ01}. We include all necessary results from these sources here for notation and ease of reference. We use $S(\Delta)$ to denote the (exterior) algebraic shifting of $\Delta$.

\begin{proposition}\label{IterBettiNumProp}\emph{\cite[Theorem 4.1]{DR00}}
Let $\Delta$ be a $d$-dimensional simplicial complex, $S(\Delta)$ its algebraic shifting, and $\beta^k[r](\Delta)$ an iterated Betti number (where $0 \le r \le k+1 \le d+1)$. Then
$$
\beta^k[r](\Delta) = \abs{\set{\text{facets}~T \in S(\Delta) ~:~ \abs{T} = k+1 ~\text{and}~\init(T)=r}}.
$$
where $\init(T) = \max\set{i \geq 0 ~:~ [i] \subseteq T}$ if $1 \in T$ and $\init(T) = 0$ otherwise.
\end{proposition}

A set $B$ of faces of a simplicial complex $\Delta$ is an \emph{$r$-Betti set} if $f_{k-r}(B) = \beta^k[r](\Delta)$ for all $k$.

\begin{theorem}\label{DZMainDecompThm}\emph{\cite[Theorem 3.2]{DZ01}}
Let $\Delta$ be a $d$-dimensional simplicial complex. Then there exists a chain of subcomplexes
$$
\{\varnothing\} = \Delta^{(d+1)} \subseteq \dots \subseteq \Delta^{(r)} \subseteq \Delta^{(r-1)} \subseteq \dots \subseteq \Delta^{(1)} \subseteq \Delta^{(0)} = \Delta,
$$
where
$$
\Delta^{(r)} = \Delta^{(r+1)} \sqcup B^{(r)} \sqcup \Omega^{(r+1)} ~~ (0 \le r \le d),
$$
and bijections
$$
\eta^{(r)} : \Delta^{(r)} \to \Omega^{(r)} ~~ (1 \le r \le d+1),
$$
such that, for each $r$,
\begin{enumerate}
\item $\Delta^{(r+1)}$ and $\Delta^{(r+1)} \sqcup B^{(r)}$ are subcomplexes of $\Delta^{(r)}$;
\item $B^{(r)}$ is an $r$-Betti set; and
\item for any $\sigma \in \Delta^{(r)}$, we have $\sigma \subsetneq \eta^{(r)}(\sigma)$ and $\abs{\eta^{(r)}(\sigma) \setminus \sigma} = 1$.
\end{enumerate}
\end{theorem}

\begin{proposition}\label{ShiftPreservesAcyc}\emph{\cite[a specialization of Theorem 4.2]{Ka00}}
If $\Delta$ is $k$-fold acyclic, then its algebraic shifting $S(\Delta)$ is also $k$-fold acyclic.
\end{proposition}

\begin{proposition}\label{Shift&AcycImpliesCone}
If $\Delta$ is shifted and $k$-fold acyclic, then $\Delta$ is a $k$-fold cone.
\end{proposition}

\begin{proof}
Since $\Delta$ is shifted, $\Delta = S(\Delta)$. By \cite[Theorem 4.3]{BK88},
$$
\beta_i(\Delta) = \abs{\set{\text{facets}~T \in \Delta ~:~ \abs{T} = i + 1 ~\text{and}~1 \not \in T}}
$$
Since $\Delta$ is assumed to be $k$-fold acyclic, it is in particular acyclic. Thus $\beta_i(\Delta) = 0$ for all $i$, which implies that $\Delta = \ideal{1} \star \Gamma_1$ for some complex $\Gamma_1$. By \cite[Proposition 2.3]{DR00}, $\Gamma_1$ is shifted on the remaining vertices, and we also know that $\Gamma_1$ is $(k-1)$-fold acyclic. 
Repeating this argument, we see that $\Delta = \ideal{1} \star \ideal{2} \star \dots \star \ideal{k} \star \Delta' = \ideal{12 \dots k} \star \Delta'$ for some subcomplex $\Delta'$, i.e., $\Delta$ is a $k$-fold cone.
\end{proof}

We are now able to prove the following relaxation of Conjecture~\ref{DiamondConjecture}.

\begin{theorem}\label{AcycBoolTreeThm}
Let $\Delta$ be $k$-fold acyclic. Then $\Delta$ can be written as a disjoint union of boolean trees of rank $k$. Furthermore, the minimal faces of these boolean trees together form a subcomplex $\Delta'$.
\end{theorem}

\begin{proof}
{The proof is similar to the proof of \cite[Corollary 3.5]{DZ01}. We will make use of Theorem~\ref{DZMainDecompThm}, and we will use the notation of that theorem.}

By Propositions \ref{ShiftPreservesAcyc} and \ref{Shift&AcycImpliesCone}, $S(\Delta) = \ideal{1\dots k} \star \Delta'$ for some complex $\Delta'$. 
Since $S(\Delta)$ is a $k$-fold cone, $\init(T) \ge k$ for all facets $T \in S(\Delta)$, and thus $\beta^i[r](\Delta) = 0$ for $r < k$ by Proposition~\ref{IterBettiNumProp}.

Step 0: Note that all faces of $\Delta = \Delta^{(0)}$ form rank $0$ boolean trees.

We will perform the following step $k$ times: Assume this step has been completed $i < k$ times, so the minimal elements of boolean trees of rank $i$ are all of the faces of $\Delta^{(i)}$. By Theorem~\ref{DZMainDecompThm},
\begin{align*}
\Delta^{(i)} &= \Delta^{(i+1)} \sqcup B^{(i)} \sqcup \Omega^{(i+1)} \\
&= \Delta^{(i+1)} \sqcup \Omega^{(i+1)}
\end{align*}
with the second equality by Theorem~\ref{DZMainDecompThm} (2) since $i<k$. For each face $\sigma \in \Delta^{(i+1)}$, we combine the rank $i$ boolean trees with minimal elements $\sigma$ and $\eta^{(i+1)}(\sigma)$ to form rank $(i+1)$ boolean trees. Since $B^{(i)} = \varnothing$, there are no rank $i$ boolean trees remaining after this step.

Furthermore, if we stop this process after $k$ iterations, we see that the minimal elements of the resulting boolean trees are precisely the faces of $\Delta^{(k+1)} \sqcup B^{(k)}$. 
We know that
$$
\Delta^{(k+1)} \sqcup B^{(k)} \subseteq \Delta^{(k)} \subseteq \Delta^{(k-1)} \subseteq \dots \subseteq \Delta^{(0)} = \Delta
$$
as subcomplexes, therefore the minimal elements of these boolean trees together form a subcomplex $\Delta' = \Delta^{(k+1)} \sqcup B^{(k)}$.
\end{proof}

The subcomplex $\Delta^{(k+1)} \sqcup B^{(k)}$ described in Theorem~\ref{AcycBoolTreeThm} is a combinatorial witness to the subcomplex in \cite[Proposition 2.3]{St93}. This shows that a correct generalization of Stanley's acyclic matching is to boolean trees rather than boolean intervals.

We note the similarity between the resolution of this conjecture and the Partitionability Conjecture (see, e.g., \cite{DK17, St79}). A complex $\Delta$ is \emph{partitionable} if its face poset can be written as the disjoint union of boolean intervals whose maximal faces are the \textit{facets} of $\Delta$. Though there exist Cohen--Macaulay complexes which are not partitionable \cite{DG16}, all Cohen--Macaulay complexes do have a similar decomposition if ``boolean interval'' is replaced in the definition of partitionable with ``boolean tree'' \cite[Theorem 5.4]{DZ01}.

\section{{\it d}-fold Acyclic Complexes}\label{dimfold}

In this section, we will show that Conjecture~\ref{DiamondConjecture} holds for $d$-fold acyclic complexes where $d = \dim \Delta$. We first show that Conjecture~\ref{DiamondConjecture} holds for stacked complexes. We then show that $d$-dimensional $d$-fold acyclic complexes must be stacked. Thus Conjecture~\ref{DiamondConjecture} holds when $k=\dim\Delta$.

Our interest in this case was sparked by the following result.

\begin{theorem}\label{2dimStacked}\emph{[Duval--Klivans--Martin, unpublished]}
If $\Delta$ is $2$-dimensional and $2$-fold acyclic, then $\Delta$ is stacked.
\end{theorem}

Theorem~\ref{2dimStacked} together with the following proposition shows that Conjecture~\ref{DiamondConjecture} holds if $\dim \Delta \le 2$.

\begin{proposition}\label{StackedIsDiamondableProp}
Let $\Delta$ be a $d$-dimensional stacked simplicial complex. Then $\Delta$ is $d$-fold acyclic and $\Delta$ can be written as the disjoint union of rank $d$ boolean intervals, the minimal elements of which form a subcomplex $\Delta' \subseteq \Delta$. Furthermore, Conjecture~\ref{DiamondConjecture} holds for stacked complexes.
\end{proposition}

\begin{proof}
We will first show that stacked complexes are $d$-fold acyclic by induction on $d$. As a base case, notice that if $\dim \Delta = 1$, then $\Delta$ is stacked if and only if $\Delta$ is a connected acyclic graph (i.e., a tree), and thus is $1$-fold acyclic.

Assume the result holds for lower dimensions. 
Let $\sigma \in \Delta$ such that $\abs{\sigma} < d$. We note that $\link_\Delta \sigma$ has a stacked shelling order induced from the stacked shelling of $\Delta$. 
If $\sigma \ne \varnothing$, then $\dim \link_\Delta \sigma < \dim \Delta$, and since \(\link_\Delta \sigma\) is stacked and of lower dimension, it is also acyclic by assumption. 
 If instead $\sigma = \varnothing,$ then $\link_\Delta \sigma = \Delta$, which is acyclic following a standard argument about the homology of shellable complexes. Thus $\Delta$ is $d$-fold acyclic by induction on dimension.

Given a stacked complex $\Delta$, its stacked shelling $F_1,\dots F_j$ gives rise to the following decomposition:
$$
\Delta = [\varnothing, F_1] \sqcup [v_2,F_2] \sqcup [v_3,F_3] \sqcup \dots \sqcup [v_j, F_j]
$$
For any vertex $v_1 \in F_1$, we can write $[\varnothing, F_1] = [\varnothing,F_1\setminus\set{v_1}] \sqcup [v_1,F_1]$. Therefore $\Delta$ can be decomposed as
$$
\Delta = [\varnothing,F_1\setminus\set{v_1}] \sqcup [v_1,F_1] \sqcup [v_2,F_2] \sqcup [v_3,F_3] \sqcup \dots \sqcup [v_j, F_j]
$$
and $\Delta' = \set{\varnothing,v_1,v_2,\dots,v_j}$ is a subcomplex of $\Delta$. For any choice of \(k \leq d\), a stacked complex can be decomposed into a refinement of the above decomposition so the parts are rank \(k\) boolean intervals and the minimal elements of these intervals form a subcomplex. One way to do this is to totally order the vertices by their order of appearance in the shelling, and take the appropriate lex-least faces to be the minimal faces. Thus Conjecture~\ref{DiamondConjecture} holds for stacked complexes.
\end{proof}

\begin{lemma}\label{ddfvector}
Let \(\Delta\) be \(d\)-dimensional and \(d\)-fold acyclic. Then the f-polynomial of \(\Delta\) is \( f(\Delta,t)=(1+t)^{d}(1+nt)\) where $n$ is the number of facets of \(\Delta\).
\end{lemma}

\begin{proof}
This follows immediately from \cite[Proposition 2.3]{St93}.
\end{proof}

\begin{lemma}\label{DimFoldAcycPureandConnected}
Let $\Delta$ be $d$-dimensional and $d$-fold acyclic. Then $\Delta$ is pure and its facet-ridge graph is connected.
\end{lemma}

\begin{proof}
Let \(\mathcal{F}(\Delta)\) denote the facet-ridge graph of \(\Delta\), the graph whose vertices are facets of $\Delta$ and whose edges are pairs of facets 
of the same dimension whose intersection is a face of dimension one smaller than each of the facets. By definition, this graph is disconnected if there are facets of different dimensions.

Suppose \(\mathcal{F}(\Delta)\) is disconnected. Let \(C_1\) be the collection of $d$-dimensional facets in one component of \(\mathcal{F}(\Delta)\). Let \(C_2\) be a different collection of facets in a component of \(\mathcal{F}(\Delta)\) so that \(I= \langle C_1 \rangle \cap \langle C_2 \rangle\) has maximum dimension among all choices of \(C_2\). Let \(\sigma\) be a facet of \(I\). In \(\mathcal{F}(\Delta)\), the components containing \(C_1\) and \(C_2\) are not connected, so every face of \(I\) must be of dimension at most $d-2$. Otherwise that face would appear in \(\mathcal{F}(\Delta)\) as either a vertex or an edge in both \(C_1\) and \(C_2\). 
In particular, \(\sigma\) is at most \((d-2)\)-dimensional. Since \(\sigma\) is in \(I\), \(\link_{\langle C_1 \rangle}\sigma\) has at least one \(1\)-dimensional face and \(\link_{\langle C_2 \rangle}\sigma\) is nonempty. Since \(\sigma\) is a facet of \(I\), the intersection \(\link_{\langle C_1 \rangle}\sigma \cap \link_{\langle C_2 \rangle}\sigma =\link_I \sigma\) is  the empty face. In fact, for any other component \(C_i\) of $\mathcal{F}(\Delta)$, \(\link_{\langle C_1 \rangle}\sigma \cap \link_{\langle C_i \rangle}\sigma\) is the empty face. Therefore \(\lk_\Delta\sigma\) is disconnected and of dimension at least \(1\). Since \(\Delta\) is \(d\)-fold acyclic, \(\lk_\Delta\sigma\) must be acyclic. This is a contradiction, so \(\mathcal{F}(\Delta)\) must be connected.

Since \(\mathcal{F}(\Delta)\) is connected, \(\Delta\) must be pure.
\end{proof}

We prove the following theorem using combinatorial methods, but we note that it is ripe for proof with other methods, including a more algebraic approach using quasi-forests, Hochster's formula, and Dirac's theorem on chordal graphs.

\begin{theorem}\label{DimfoldtoStacked}
Let \(\Delta\) be \(d\)-dimensional and \(d\)-fold acyclic. Then \(\Delta\) is a stacked complex.
\end{theorem}
\begin{proof}

Let $T$ be a spanning tree of $\mathcal{F}(\Delta)$. Assume $\Delta$ has $j$ facets and order its facets $F_1,\dots,F_j$ so that $T$ restricted to $F_1,\dots,F_{i}$ is connected for all $i \in [j].$ We will show that this order is a stacked shelling of $\Delta$.

Each edge in $T$ corresponds to a ridge of $\Delta$. For each $1<i \le j$, $T$ restricted to $F_1 , \dots F_{i}$ differs from $T$ restricted to $F_1,\dots,F_{i-1}$ by only the vertex $F_{i}$ and the edge that connects it to the rest of the tree. We will call this edge $R_{i}$ and note that $R_{i} = F_{i} \cap F_{\ell}$ for some $\ell < i$. We derive from this a collection of intervals in $\Delta$. One special interval is \([\varnothing,F_1]\). The rest of the intervals are \([v_i,F_i]\), where $v_i = F_i \setminus R_i$ for $1 < i \le j$.

Let $i<j$, let \(\sigma\) be a face in \(\langle F_1, \ldots, F_i \rangle\), and let \(F_m\) be the first facet in the facet order which contains \(\sigma\). We note that \(R_m\) is contained in \(F_\ell\) for some $\ell < m$, and \(\sigma \not \subseteq F_\ell\) for any \(\ell < m\). Therefore it must be that \(\sigma \not \in [\varnothing, R_m]\) and instead \(\sigma \in [v_m,F_m]\). This means that any initial collection of intervals contains the complex generated by the corresponding facets.

Since there are \(j\) facets in the total facet order, this gives a formula for the sum of the $f$-polynomials of each interval as \((1+t)^{d+1}+(1+t)^{d}((j-1)t)\), which simplifies to \((1+t)^{d}(1+jt)\). By Lemma~\ref{ddfvector}, this is exactly the $f$-polynomial of \(\Delta\). By the converse of the addition principle, the intervals must be disjoint.

The facet order \(F_1, \ldots, F_j\) determines a collection of intervals such that the bottom element of each interval is a vertex, the intervals are disjoint, and any initial segment is the complex generated by those facets. Therefore \(F_1, \ldots, F_j\) is a stacked shelling order, and \(\Delta\) must be a stacked complex.
\end{proof}

Applying Proposition~\ref{StackedIsDiamondableProp} and Theorem~\ref{DimfoldtoStacked}, we see that a $d$-dimensional complex $\Delta$ is stacked if and only if it is $d$-fold acyclic. This leads immediately to our main result of this section.

\begin{corollary}\label{dCor}
Conjecture~\ref{DiamondConjecture} holds when $k = \dim\Delta$.
\end{corollary}

\section{Open Questions}\label{open}

While our construction gives a counterexample to Conjecture~\ref{DiamondConjecture}, our result in Theorem~\ref{AcycBoolTreeThm} provides an explicit witness to the structure of the $f$-polynomials of $k$-fold acyclic complexes. Perhaps the most interesting questions in light of Remark~\ref{SmallCounterexample} are in determining any additional conditions that would make the conjecture hold. 
We know that \(\Omega_3\) is the lowest dimensional counterexample possible, but we have no reason to suspect that is in other senses the smallest.

\begin{question}\label{smaller_fvec}
What is a minimal counterexample to Conjecture~\ref{DiamondConjecture} with respect to the total number of faces, vertices, or facets, respectively?
\end{question}
Though our counterexample is three-dimensional, it cannot be embedded into $\Rr^3$. It is unknown if non-embeddability is necessary to be a counterexample.

\begin{question}\label{embed}
Is it possible to find a counterexample to Conjecture~\ref{DiamondConjecture} that embeds into $\Rr^3$? In general, is it possible to find a $d$-dimensional counterexample that embeds into $\Rr^d$?
\end{question}
It is also unknown whether complexes with additional topological or combinatorial structure could be counterexamples.

\begin{question}\label{ball}
Do all $k$-fold acyclic simplicial balls have a rank $k$ boolean interval decomposition? If they do, must there be a decomposition so that the bottoms of these intervals forms a subcomplex?
\end{question}
Although a bit further afield from the techniques developed in this paper, one can ask about random simplicial complexes.

\begin{question}
For a fixed triple of \(k,d,v\), there are finite \(k\)-fold acyclic complexes of dimension \(d\) with \(v\) vertices. Sampling from this set with the uniform distribution, what is the probability the chosen complex has a rank \(k\) boolean decomposition? What is the limiting probability as \(v\) goes towards \(\infty\)?
\end{question}

\section*{Acknowledgements}

The second author would like to thank Jeremy Martin for all his support and guidance over the past several years.

Both authors thank Steve Klee, Isabella Novik, John Shareshian, and Matt Stamps for helpful conversations about the presentation of these results. We are especially grateful to Christos Athanasiadis, Justin Lyle, Alex McDonough, and Lei Xue, who gave invaluable comments and suggestions on earlier drafts of this paper. We also thank the anonymous referees for their careful reading and helpful remarks.

The open-source software Sage \cite{sage} was used throughout this project.

\bibliographystyle{amsplain}

\providecommand{\bysame}{\leavevmode\hbox to3em{\hrulefill}\thinspace}
\providecommand{\MR}{\relax\ifhmode\unskip\space\fi MR }
\providecommand{\MRhref}[2]{%
  \href{http://www.ams.org/mathscinet-getitem?mr=#1}{#2}
}
\providecommand{\href}[2]{#2}

\bibliography{ourbib}

\end{document}